\documentclass[12pt]{amsart}
\usepackage{amsmath,amsthm,amssymb,color,enumerate}
\usepackage{graphicx}
\usepackage[normalem]{ulem}
\usepackage{amsrefs}
\usepackage{hyperref}
\usepackage{float}
\usepackage{caption}
\usepackage{subcaption}
\usepackage{amscd}
\usepackage{longtable}
\usepackage{graphicx}
\usepackage{comment}
\usepackage{tikz-cd}
\hypersetup{colorlinks=true,linkcolor=blue,citecolor=blue}

\usepackage{mathtools}
\usetikzlibrary{cd}
\usetikzlibrary{decorations.pathmorphing}

\usepackage{amscd}

\newcommand{\rank}{\mathrm{rank}}
\newcommand{\diag}{\mathrm{diag}}

\newtheorem{theorem}{Theorem}[section]
\newtheorem{proposition}[theorem]{Proposition}
\newtheorem{lemma}[theorem]{Lemma}
\newtheorem{definition}[theorem]{Definition}
\newtheorem{conjecture}[theorem]{Conjecture}
\newtheorem{example}[theorem]{Example}

\renewcommand{\caption}[1]{\singlespacing\hangcaption{#1}\normalspacing}

\title {The Tensor Rank Problem over the Quaternions}
\author {YG Liang, Sergio Da Silva, Yang Zhang}

\address{YG Liang, University of Manitoba, Winnipeg MB}
\email{liangy1@myumanitoba.ca}

\address{Sergio Da Silva, McMaster University, Hamilton ON}
\email{smd322@cornell.edu, dasils19@mcmaster.ca}

\address{Yang Zhang, University of Manitoba, Winnipeg MB}
\email{Yang.Zhang@umanitoba.ca}

\thanks{Research supported in part by a PIMS postdoctoral fellowship, Canada NSERC and UM Interdisciplinary/New
Directions Research Collaboration Initiation Grants}

\begin{document}

\begin{abstract}
We provide a nontrivial bound on the rank of any tensor $T$ over the quaternions $\mathbb{H}$ in the $n_1\times n_2\times n_3$ cases where $2\leq n_i\leq 3$. We describe a decomposition of $T$ into $3$ simple tensors in the $2\times 2\times 2$ case. We also show that the upper bound is the best possible for some of the cases, and we provide various partial results involving tensor decompositions over $\mathbb{C}$ and $\mathbb{H}$.
\end{abstract}\bigskip

\maketitle

Tensors, as generalizations of matrices to higher dimensions, have many applications in various settings, such as aerospace engineering (\cite{doostan2007least}), signal processing
(\cite{comon2002tensor, de2007tensor, de2009survey}), data mining (\cite{liu2005text,   sun2006window}), machine learning (\cite{RSG}), computer vision (\cite{shashua2005non, vasilescu2002multilinear, vasilescu2003multilinear}), higher-order statistics (\cite{comon1994independent, comon1996decomposition}), pattern
recognition (\cite{kim2007nonnegative, savas2007handwritten}), chemometrics  (\cite{comon2009tensor,smilde2005multi}), graph analysis (see \cite{kolda2005higher}), numerical linear algebra
(\cite{de2000multilinear,de2000best,kolda2001orthogonal}), numerical analysis (Part I in \cite{Landsberg}), etc.

There have been extensive studies for tensor decompositions and tensor ranks. An overview of the theoretical developments and applications of tensor decompositions can be found in \cite{kolda2009tensor}. For any given tensor, finding a decomposition or determining its rank can generally be a difficult question. In contrast to tensor decomposition and tensor rank problems over conventional algebras where there are some known results for specific cases over the complex and real numbers, the tensor decomposition and tensor rank problems over the real quaternion algebra
\[\mathbb{H} = \{a_0+a_1i+a_2j+a_3k|i^2=j^2=k^2=ijk=-1; a_0,a_1,a_2,a_3\in\mathbb{R}\}.\]
are at present far from fully developed and remain a largely open question. In this article, we will consider tensors of the form
\[\mathbb{H}^{n_1}\otimes_{\mathbb{H}}\mathbb{H}^{n_2}\otimes_{\mathbb{H}}\mathbb{H}^{n_3},\]
which is the set of multiway arrays in $\mathbb{H}^{n_1\times n_2\times n_3}$ that has an $M_{n_1}(\mathbb{H})\times M_{n_2}(\mathbb{R})\times M_{n_3}(\mathbb{H})$ action defined by $(a_1,a_2,a_3)\cdot (h_1,h_2,h_3) = (a_1h_1,a_2h_2,h_3a_3)$. In other words, we are considering $(M_{n_1}(\mathbb{H}),M_{n_3}(\mathbb{H}))$-bimodules and have an additional real action on frontal slices. The tensor rank problem asks how to minimally decompose $T\in \mathbb{H}^{n_1}\otimes_{\mathbb{H}}\mathbb{H}^{n_2}\otimes_{\mathbb{H}}\mathbb{H}^{n_3}$ into a sum of simple tensors. This can involve determining the rank of $T$, but also includes questions about the existence and uniqueness of some minimal decomposition. We address some of these questions for the cases where $n_i \leq 3$. 

Knowing how to check the rank of a given tensor and finding a minimal decomposition into simple tensors is not just of theoretical importance, but has many real life applications as well. Such results can be useful for questions arising in applied mathematics, engineering, physics and computer science. For example, Sylvester-type equations for tensors usually involves assumptions that depend on tensor rank. Not all problems are commutative in nature however, so understanding tensors in the noncommutative case is also important. 

There are very few results involving tensor decompositions  over $\mathbb{H}$, at least in comparison to known results over $\mathbb{C}$ and $\mathbb{R}$. In \cite{Rank2}, upper bounds for the rank of tensors with size $2\times\cdots\times 2$ over $\mathbb{R}$ or $\mathbb{C}$ are studied. For instance, the maximal rank of any real $2\times 2\times 2\times 2$ tensor is 5, while the maximal rank for a complex tensor of the same size is 4 (these bounds were shown earlier in \cite{Kong-Jiang} and \cite{Brylinski} respectively). In some recent work involving $\mathbb{H}$, a simulatneous diagonalization result in \cite{Zhang2018} produces solutions to a specific generalized Sylvester quaternion matrix equation, while expanded work in \cite{Zhang2019} provides solutions for a two-sided coupled Sylvester-type equation in a similar setting. These results however assume that the rank of the tensors being used are known. It is therefore essential to find ways of determining the rank of tensors over $\mathbb{H}$ if one hopes to utilize these results. Our goal is to provide explicit criteria which can easily be checked via a computer program, as opposed to criteria which theoretically determines the rank of a given tensor.

Our results begin by analyzing the $2\times 2\times 2$ case. It is known that over the complex and real numbers, the maximum rank of any tensor is 3 (see \cite{SMS}). In the quaternion case, we show that maximum rank is also 3, and we provide a convenient decompositon into simple tensors. We will then provide a bound on the rank for tensors over $\mathbb{H}$ in the remaining $n_1\times n_2\times n_3$ cases, where $2\leq n_i\leq 3$. Finally, we provide results involving the decompositions for some these tensors, as well as examples of tensors which achieve the maximal bound. The results on tensor rank bounds can be summarized with the following theorem.

\begin{theorem}
Let $T\in \mathbb{H}^{n_1}\otimes_{\mathbb{H}}\mathbb{H}^{n_2}\otimes_{\mathbb{H}}\mathbb{H}^{n_3}$ where $2\leq n_i\leq 3$. Then in the $2\times 2\times 2$ case, there exists an explicit decompositon of $T$ into $3$ simple tensors, so $\rank(T)\leq 3$. Furthermore, $\rank(T)\leq 3$ if $n_i=3$ for exactly one $i$, $\rank(T)\leq 4$ if $n_i=2$ for exactly one $i$, and $\rank(T)\leq 6$ in the $3\times 3\times 3$ case.
\end{theorem}

\section{Preliminaries}\label{definitions}

A multiway array $T=(T_{i_{1}i_{2}\ldots i_{K}})$ where $1\leq i_{1}\leq N_{1},\ldots, 1\leq i_{K}\leq N_{K}$ is called a $K$-way tensor of size $(N_{1},N_{2},\ldots,N_{K})$. We also say that $T$ is an $N_{1}\times N_{2}\times\ldots\times N_{K}$ tensor. When $k=3$, we will use the convention that $N_2$ indicates the number of \textit{frontal slices} of the array, so that the array consists of $N_2$ many $N_1\times N_3$ matrices. Similarly, it will have $N_1$ \textit{horiztonal slices} and $N_3$ \textit{lateral slices}. For example a $3\times 2\times 3$ tensor has 18 entries, which can be denoted by

\[
   T=\left(
  \begin{bmatrix}
    a_{111} & a_{121} & a_{131}\\
   a_{211} & a_{221}  & a_{231}\\
   a_{311} & a_{321}  & a_{331}\\
  \end{bmatrix};
  \begin{bmatrix}
    a_{112} & a_{122} & a_{132}\\
   a_{212} & a_{222}  & a_{232}\\
   a_{312} & a_{322}  & a_{332}\\
  \end{bmatrix}\right)
\]
or
\[
   T=\left(
  \begin{bmatrix}
\vec{a}_{11} &  \vec{a}_{12}&  \vec{a}_{13}\\
\vec{a}_{21} &  \vec{a}_{22}&  \vec{a}_{23}\\
\vec{a}_{31} &  \vec{a}_{32}&  \vec{a}_{33}\\
  \end{bmatrix}\right) ,
\]
where $\vec{a}_{ij}=(a_{ij1},a_{ij2})$, $i,j=1,2,3.$ Here the first lateral slice is the $3\times 2$ matrix defined by the vectors $\vec{a}_{i1}$, and the first horizontal slice is the $2\times 3$ matrix defined by the $\vec{a}_{1j}$. \newline

\noindent\textbf{Note:} We should mention that this notation is different from the one used by some authors cited in this article (for example, in \cite{SMS2}, $N_1$ indicates the number of frontal slices).

\begin{definition}
A nonzero $K$-tensor $T=(T_{i_{1}i_{2}\ldots i_{K}})$ is called a \textbf{simple tensor} if there exist vectors 
\begin{align*}
&\vec{a_{1}}=(a_{11},a_{12},\ldots,a_{1N_{1}}),\\
&\vec{a_{2}}=(a_{21},a_{22},\ldots,a_{2N_{2}}),\\
& \hspace{2.5cm}\vdots\\
&\vec{a_{K}}=(a_{K1},a_{K2},\ldots,a_{KN_{K}}),
\end{align*}
such that $T=(T_{i_{1}i_{2}\ldots i_{K}})=(a_{1i_{1}}a_{2i_{2}}\ldots a_{Ki_{K}})$. We will also denote this by 
\[T=\vec{a}_{1}\otimes \vec{a}_{2}\otimes\dots\otimes \vec{a}_{K}.\]
\end{definition}

\begin{example}\normalfont
For the $2\times 3\times 2$ tensor
\[
   T=\left(
  \begin{bmatrix}
    2 & 3  \\
   8 & 12  \\
  \end{bmatrix};
    \begin{bmatrix}
    -2 & -3  \\
   -8 & -12  \\
  \end{bmatrix};
    \begin{bmatrix}
    4 & 6  \\
   16 & 24  \\
  \end{bmatrix}\right) ,
\]
there exist vectors \[\vec{a}=(a_{1},a_{2})=(1,4),\hspace{2mm} \vec{b}=(b_{1},b_{2},b_{3})=(1,-1,2) \text{ and } 
\vec{c}=(c_{1},c_{2})=(2,3)\] such that 
\begin{align*}
   T&=\left(\begin{bmatrix}
    1\times 1\times 2 & 1\times 1\times 3  \\
     4\times 1\times 2 & 4\times 1\times 3  \\
  \end{bmatrix};
    \begin{bmatrix}
    1\times (-1)\times 2 & 1\times (-1)\times 3  \\
     4\times (-1)\times 2 & 4\times (-1)\times 3  \\
  \end{bmatrix};
    \begin{bmatrix}
    1\times 2\times 2 & 1\times 2\times 3  \\
     4\times 2\times 2 & 4\times 2\times 3  \\
  \end{bmatrix}\right) \\
    &=\left(\begin{bmatrix}
    a_{1}b_{1}c_{1} & a_{1}b_{1}c_{2}  \\
    a_{2}b_{1}c_{1} & a_{2}b_{1}c_{2}  \\
  \end{bmatrix};
    \begin{bmatrix}
    a_{1}b_{2}c_{1} & a_{1}b_{2}c_{2}  \\
    a_{2}b_{2}c_{1} & a_{2}b_{2}c_{2}  \\
  \end{bmatrix};
    \begin{bmatrix}
    a_{1}b_{3}c_{1} & a_{1}b_{3}c_{2}  \\
    a_{2}b_{3}c_{1} & a_{2}b_{3}c_{2}  \\
  \end{bmatrix}\right) \\
  &=  \vec{a}\otimes \vec{b}\otimes \vec{c}.
\end{align*}
Therefore, $T$ is a simple tensor. \hfill$\Box$
\end{example}
\begin{definition}
Let $T$ be a nonzero $K$-tensor. Then the rank of $T$ is the smallest positive integer $n$ such that  $T=T_{1}+T_{2}+\ldots+T_{n}$ where  $T_1,T_2,\ldots, T_n$ are simple $K$-tensors. We will say that $T$ has rank $n$ and denote this by $\rank(T)=n$. 
\end{definition}
An immediate consequence of this definition is that $\rank(T+S) \leq \rank(T) + \rank(S)$. Any sum of simple tensors $T_1+T_2+\ldots+T_n = T$ is a called \textbf{tensor decomposition} for $T$, even if $n$ is not minimal. We will highlight various nontrivial tensor decompositions in subsequent sections where $n$ is always equal to the minimal known bound on tensor rank. It should be noted that some authors refer to a tensor decomposition in the singular value decomposition sense, and not as a sum of simple tensors (see \cite{HNW} for this type of tensor decomposition in the quaternion case).

In many of the results in this article, it is desirable to first simplify a tensor $T=(A_1;\ldots;A_k)$ by first applying column and row operations to the matrices $A_i$ which preserve rank. Since we are working over a noncommutative division ring, we need to be careful with how $\mathbb{H}$ is acting in a column or row operation (see \cite[Section 1.3.3]{Widdows} for an exposition on the difficulties of defining tensors over the quaternions).  For example, in applying the two column operations shown below, we have increased the rank of the matrix:

\[
  \begin{bmatrix}
    i & i  \\
   i+j & i+j  \\
  \end{bmatrix}\xrightarrow[C_1\rightarrow jC_1]{C_2\rightarrow C_2j}
  \begin{bmatrix}
   -k & k  \\
   -1-k & -1+k  \\
  \end{bmatrix}.
  \]
  
 \hspace{5mm}
  
If we however only act by $\mathbb{H}$ on the left when using row operations (ie. horizontal slice operations), and by $\mathbb{H}$ on the right for column operations (ie. lateral slice operations), then one can check that the rank is preserved. In particular, we are actually endowing $\mathbb{H}^{n_1\times\ldots\times n_k}$ with a bimodule structure where multiplation takes place on the left and right by nonsingular quaternionic matrices. Using real frontal slice operations is also allowed since $\mathbb{R}$ is the center of $\mathbb{H}$. We will call each of these \textbf{rank-preserving operations}. In the latter case for example, let $A_{i}$ be an $N_{i}\times M_{i}$ matrix with entries in $\mathbb{R}$ for $i=1,\ldots,p$. Consider the multilinear map defined by
\begin{align*}
    A_{1}\otimes \ldots \otimes A_{p}: \mathbb{H}^{N_{1}\times \ldots\times N_{p}}&\longrightarrow \mathbb{H}^{M_{1}\times \ldots\times M_{p}},\\
    \vec{v}_{1}\otimes \ldots \otimes \vec{v}_{p}& \mapsto A_{1}\vec{v}_{1}\otimes \ldots \otimes A_{p}\vec{v}_{p}.
\end{align*}
This map is well-defined since the action is linear in each component. When we apply $A_{1}\otimes \ldots \otimes A_{p}$ to $T$, we can bound the rank of the resulting image by $\rank(T)$.

\begin{lemma}\label{multilinear}
Let $T$ be an $N_{1}\times\ldots\times N_{p}$ tensor and $A_{1}\otimes \ldots \otimes A_{p}$ a multilinear map from $\mathbb{H}^{N_{1}\times \ldots\times N_{p}}$ to $\mathbb{H}^{M_{1}\times \ldots\times M_{p}}$ defined as above. Then $\rank((A_{1}\otimes \ldots \otimes A_{p})(T))\leq \rank(T)$. Furthermore, if $A_{1}\otimes \ldots \otimes A_{p}$ is invertible, then $\rank((A_{1}\otimes \ldots \otimes A_{p})(T))=\rank(T)$.
\end{lemma}
\begin{proof}
Suppose $\rank(T)=n$, and write $T$ as a sum of simple tensors
\[ T=\displaystyle\sum_{i=1}^{n}v_{i_1}\otimes \ldots \otimes v_{i_p}.\]
By multilinearity, we have
 \[\displaystyle (A_{1}\otimes \ldots \otimes A_{p})(T)=\sum_{i=1}^{n}A_{1}v_{i_1}\otimes \ldots \otimes A_{p}v_{i_p},\]
which implies $\rank((A_{1}\otimes \ldots \otimes A_{p})(T))\leq n$.
\end{proof}

A similar proof works if we replace $\mathbb{H}$ with $\mathbb{C}$. Allowing the $A_i$ to have entries in $\mathbb{H}$ would no longer define a multilinear map. However, using row operations with left $\mathbb{H}$ multiplication (or column operations with right $\mathbb{H}$ mutliplication) coming from nonsingular matrices does preserve rank. 

Let us fix the notation for the adjoint of a quaternion matrix. Given an $n\times n$ matrix $A$ with entries in $\mathbb{H}$, we can uniquely write $A = A_1+A_2j$, where $A_1$ and $A_2$ are $n\times n$ matrices with entires in $\mathbb{C}$. The \textbf{complex adjoint matrix} of $A$ (or simply the adjoint of $A$), is defined as the $2n\times 2n$ complex block matrix

\[\chi_A=
\left[
\begin{array}{c c}
A_1 & A_2

\vspace{2mm}\\

-\overline{A_2} & \overline{A_1}\\
\end{array}
\right].
\]

\vspace{3mm}

\noindent The adjoint matrix is very useful in converting a diagonalization problem over $\mathbb{H}$ into a diagonalization problem over $\mathbb{C}$. The following result is well-known (\cite{Rodman}).

\begin{lemma}\label{adjoint}
An $n\times n$ matrix $A$ with entries in $\mathbb{H}$ is diagonalizable if and only if $\chi_A$ is diagonalizable.
\end{lemma}

We conclude this section with two auxiliary lemmas which are important in the sections that follow.

\begin{lemma}\label{dia1}
Let $T=(A_{1};A_{2};\ldots;A_{p})$ be an $m \times p \times n$ tensor. Then $\rank(T)\leq r$ if and only if there are $r\times r$ diagonal matrices $D_{i}$, an $m\times r$ matrix $P$, and an $r\times n$ matrix $Q$ such that $A_{k}=PD_{k}Q,$ for $k=1,\ldots, p.$
\end{lemma}

\begin{proof}
The following argument is a slight alteration from the one found in \cite[Proposition 2.1]{SMS2}. First suppose that $\rank(T)\leq r$, and write $T$ as a sum of simple tensors
\[\displaystyle T=\sum_{i=1}^{r} \vec{a_{i}}\otimes \vec{b_{i}}\otimes \vec{c_{i}}.\]
Let us write $\vec{b_{i}}=(b_{i1},b_{i2},\ldots, b_{ip})$ for $1\leq i\leq r$, so that  $\displaystyle A_{k}=\sum_{i=1}^{r} \vec{a_{i}}b_{ik} {\vec{c_{i}}}^{T}$.
We will define $P,Q$ and $D_k$ by 
 \[ P=[\vec{a_{1}},\vec{a_{2}},\ldots, \vec{a_{r}}]\text{,\hspace{4mm}}
   Q=
  \begin{bmatrix}
    {\vec{c_{1}}}^{T}  \\
   {\vec{c_{2}}}^{T}  \\
   \vdots  \\
   {\vec{c_{r}}}^{T}  
  \end{bmatrix}\text{,\hspace{4mm}}
  D_{k}=\diag(b_{1k},b_{2k},\ldots, b_{rk}).
\]
 Then
 \[
   PD_{k}Q=[\vec{a_{1}},\vec{a_{2}},\ldots, \vec{a_{r}}]
   \begin{bmatrix}
    b_{1k} & & \\
    &b_{2k} & & \\
    && \ddots & \\
    && & b_{rk}
  \end{bmatrix}
  \begin{bmatrix}
    {\vec{c_{1}}}^{T}  \\
   {\vec{c_{2}}}^{T}  \\
   \vdots  \\
   {\vec{c_{r}}}^{T}  
  \end{bmatrix}
  =\sum_{i=1}^{r} \vec{a_{i}}b_{ik} {\vec{c_{i}}}^{T}=A_{k},
\]
as required.

Proceeding similarly for the other direction, assume that there are $r\times r$ diagonal matrices $D_{k}$, an $m\times r$ matrix $P$, and an $r\times n$ matrix $Q$ such that $A_{k}=PD_{k}Q,$ for $k=1,\ldots, p.$
Then writing $P$, $Q$ and $D_k$ as above, we have:
\[\displaystyle A_{k}=PD_{k}Q=\sum_{i=1}^{r} \vec{a_{i}}b_{ik} {\vec{c_{i}}}^{T}\Longrightarrow T=\sum_{i=1}^{r} \vec{a_{i}}\otimes \vec{b_{i}}\otimes \vec{c_{i}}.\]
Hence $\rank(T)\leq r$, completing the proof.
\end{proof}

\begin{lemma}\label{dia2}
Let $T=(A_{1};A_{2};\ldots;A_{p})$ be an $n\times p\times n$ tensor, where $A_{1}$ is nonsingular. Then $\rank(T)=n$ if and only if \{$A_{j}A_{1}^{-1}\mid j= 2,3,\ldots,p$\} can be simultaneously diagonalized.
\end{lemma}

\begin{proof}

The following proof is similar to that found in \cite[Proposition 2.5]{SMS2}. First suppose that $\rank(T)=n$. Then by Lemma \ref{dia1}, there exists an $n\times r$ matrix $P$, an $r\times n$ matrix $Q$ and $r\times r$ diagonal matrices $D_{1}, D_{2},\ldots,D_{p}$ such that 
\[A_{1}=PD_{1}Q,\hspace{2mm} A_{2}=PD_{2}Q,\hspace{2mm}  \ldots, \hspace{2mm} A_{p}=PD_{p}Q.\]
Since $A_{1}$ is non-singular, $D_{1}$ must have rank $n$ and thus $r=n$. This implies that $P$, $Q$ and $D_{1}$ are each invertible. Therefore 
\[A_{j}A_{1}^{-1}=PD_{j}Q(PD_{1}Q)^{-1}=PD_{j}QQ^{-1}D_{1}^{-1}P^{-1}=PD_{j}D_{1}^{-1}P^{-1}\]
for $j=1,\ldots n$, and thus the $A_jA_1^{-1}$ can be simultaneously diagonalized.

For the other direction, suppose that there exists an $n\times n$ matrix $P$ where 
\[D_j=P^{-1}A_{j}A_{1}^{-1}P, \hspace{2mm} j=2,3,\ldots,p\]
are diagonal matrices. Consider the tensor
\[T^{\prime}=P^{-1}TA_{1}^{-1}P=(I_{n};D_{2};\ldots;D_{p}).\]
It is easy to check that  $\rank(T^{\prime})= n$, and since multiplication by invertible matrices is rank preserving, $\rank(T)=n$ as required. 
\end{proof}

\section{The $2\times 2\times 2$ case}

We will show in this section that $2\times 2\times 2$ quaternion tensors have a rank no greater than 3. A priori, we could guarantee a trivial bound of 4, so it would be useful to start with a motivating argument on why an attempt to reduce this bound to 3 is justified. Considering the tensor rank problem from the algebro-geometric perspective, it is natural to try and compute the \textbf{generic rank}, which is the minimum $r$ such that the set of all tensors of rank at most $r$ is a Zariski dense set in $\mathbb{H}^{2}\otimes_{\mathbb{H}}\mathbb{H}^{2}\otimes_{\mathbb{H}}\mathbb{H}^{2}$.

In the complex case, one can show that the generic rank is 2 using an argument from \cite[Proposition 12.4.3.2]{Landsberg}. We can view $T\in A\otimes B\otimes C = \mathbb{C}^{2}\otimes_{\mathbb{C}}\mathbb{C}^{2}\otimes_{\mathbb{C}}\mathbb{C}^{2}$ as a map from $A^*\rightarrow B\otimes C$. Then $\rank(T)$ is the number of rank one matrices needed to span $T(A^*)\subset B\otimes C$ as a vector space (see \cite[Theorem 3.1.1.1]{Landsberg}). We can projectivize to get $\mathbb{P}T(A^*)\subset \mathbb{P}(B\otimes C)$ and consider the Segre embedding

\[ \sigma: \mathbb{CP}^1\times\mathbb{CP}^1\longrightarrow \mathbb{CP}^3\]
\[[a,b]\times[c,d]\mapsto [ac,ad,bc,bd]=[w,x,y,z]. \]

\hspace{3mm}

The space of simple tensors is isomorphic to $\sigma( \mathbb{CP}^1\times\mathbb{CP}^1),$ which is a projective subvariety of $\mathbb{CP}^3 = \mathbb{P}(B\otimes C)$. It is not difficult to show that $wz = xy$ generates the ideal of the image of $\sigma$. Notice that $\mathbb{P}T(A^*)$ is generally a degree 1 hypersurface in $\mathbb{P}(B\otimes C)$ since $\mathbb{P}T(A^*)$ is a linear subspace. It will intersect $\sigma( \mathbb{CP}^1\times\mathbb{CP}^1)$ at deg($\sigma( \mathbb{CP}^1\times\mathbb{CP}^1))=2$ many points by B\'{e}zout's theorem. Since these points correspond to simple tensors in the preimage, we can write $T$ as the sum of 3 simple tensors. This shows that the generic rank in the complex case is 2.

If we try and apply the same argument in the quaternion case, we immediately run into problems. We could  assume that $w,x,y$ and $z$ are not commutative so that the equation that they satisfy is $wy^{-1} = xz^{-1}$ which can be simplified to $wy^{-1}z = x$. If we restrict to the Zariski open set where $y$ is invertible, we are left with a degree 3 polynomial. This would suggest that the generic rank in the noncommutative case is 3. However, making this formal proves to be more difficult than the complex case. 

The Segre map in general is defined on vector spaces by $\mathbb{P}(V)\times \mathbb{P}(W)\longrightarrow \mathbb{P}(V\otimes W)$. We could extend this idea to quaternionic vector spaces so that $V\times W = \mathbb{HP}^m\times \mathbb{HP}^n $, except that $V\otimes W$ would no longer be quaternionic vector space (see \cite[Section 1.3.3]{Widdows}), so $\mathbb{P}(V\otimes_{\mathbb{H}} W)$ is ill-defined. However, we could use the tensor product over the complex numbers, and replace $\mathbb{P}(V\otimes_{\mathbb{H}} W)$ with the Grassmannian $\text{Gr}_2(V\otimes_{\mathbb{C}} W)$. All of this taken together suggests that the generic rank over $\mathbb{H}$ is 3.

\begin{conjecture}
The generic rank of a tensor in $\mathbb{H}^{2}\otimes_{\mathbb{H}}\mathbb{H}^{2}\otimes_{\mathbb{H}}\mathbb{H}^{2}$ is 3.
\end{conjecture}

With this result in hand, it makes sense to try and prove that any $2\times 2\times 2$ quaternion tensor has rank at most 3. The authors first proved this bound by writing out a system of polynomial equations that a tensor having rank 3 must satisfy. Using some basic results from algebraic geometry, we were able to prove the existence of a solution. This allowed us to arrive at the explicit and more convenient decomposition found below. 

\begin{proposition}\label{222decomp}
Let \[
   T=\left(
  \begin{bmatrix}
    A_{11} & A_{12}  \\
    A_{21} & A_{22}  \\
  \end{bmatrix};
  \begin{bmatrix}
    B_{11} & B_{12}  \\
    B_{21} & B_{22}  \\
  \end{bmatrix}\right) 
\] be a $2\times 2\times 2$ quaternion tensor. If $A_{11}B_{11}\neq 0$, then $T=T_{1}+T_{2}+T_{3}$ where:

\begin{align*}
& T_{1}=\left(
  \begin{bmatrix}
    A_{11} & A_{11}(A_{11}^{-1}A_{12}) \\
    A_{21} & A_{21}(A_{11}^{-1}A_{12}) \\
  \end{bmatrix};
  \begin{bmatrix}
    0 & 0 \\
    0 & 0 \\
  \end{bmatrix}\right),\\
& T_{2}=\left(
  \begin{bmatrix}
    0 & 0  \\
   0 & 0 \\
  \end{bmatrix};
  \begin{bmatrix}
B_{11} & B_{11}(B_{11}^{-1}B_{12})\\
B_{21} & B_{21}(B_{11}^{-1}B_{12}) \\
  \end{bmatrix}\right),\\
&   T_{3}=\left(
  \begin{bmatrix}
    0 & 0  \\
   0 & A_{22}-A_{21}(A_{11}^{-1}A_{12})  \\
  \end{bmatrix};
  \begin{bmatrix}
    0 & 0  \\
   0 & B_{22}-B_{21}(B_{11}^{-1}B_{12})  \\
  \end{bmatrix}\right).\\
\end{align*}
\end{proposition}

One could ask whether this tensor decomposition is unique up to rescaling or permutation indeterminacy (see \cite{RSG} for details about uniqueness of tensor decompositions). While it is not difficult to find examples where a decomposition does not appear to be unique, we will not consider uniqueness questions for the purposes of this article. We can however bound the rank of any quaternion tensor $T$ of size $2\times 2\times 2$ using the explicit decomposition into at most 3 simple tensors from the proposition. To ensure that we can apply this result, we first need to ensure that we can write $T$ in the desired form, and this requires us to use row and column operations.

\begin{theorem}\label{222quaternion}
Let $T$ be a $2\times 2\times 2$ quaternion tensor. Then $\rank(T)\leq 3$.
\end{theorem}
\begin{proof}
Let us write the tensor as  \[
   T=\left(
  \begin{bmatrix}
    A_{11} & A_{12}  \\
    A_{21} & A_{22}  \\
  \end{bmatrix};
  \begin{bmatrix}
    B_{11} & B_{12}  \\
    B_{21} & B_{22}  \\
  \end{bmatrix}\right) .
\] Since $\mathbb{R}$ is the center of $\mathbb{H}$, by Lemma \ref{multilinear} we can perform real elementary operations on $T$ (including frontal slice, horizontal slice and lateral slice operations). Except for the trivial cases with too many zero entries (in which case the tensor rank is obvious), we may assume that $A_{11}$ and $B_{11}$ are not zero. Then $\rank(T)\leq 3$ follows from Proposition \ref{222decomp} since the rank of each $T_i$ is no more than 1.
\end{proof}

It is not difficult to show that the tensor

\[
   T=\left(
  \begin{bmatrix}
1 & 0  \\
0 & 1  \\
  \end{bmatrix};
      \begin{bmatrix}
0 & 1  \\
0 & 0  \\
 \end{bmatrix}\right)
     \]

\vspace{3mm}

\noindent has rank 3. In fact, it follows immediately from Lemma \ref{dia2}. Therefore, the bound on the rank for $2\times 2\times 2$ quaternion tensors is the best possible.

\section{The $2\times 2\times 3$ and $2\times 3\times 2$ cases}

When working over a commutative ring, there is no difference between the $2\times 2\times 3$, $2\times 3\times 2$, and $3\times 2\times 2$ tensor cases. Over the quaternions however, the $2\times 3\times 2$ tensor differs from the other two cases. 

We will start our discussion by working over the complex numbers. While it is known that the rank of a complex tensor of this size is bounded by 3, as far as the authors are aware, an explicit decomposition like the one below has not been made readily available. 

\begin{theorem}\label{232complex}
Let $T$ be a complex $2\times 3\times 2$ tensor

 \[
   T=\left(
  \begin{bmatrix}
    A_{1} & A_{2} \\
   A_{3} & A_{4} \\
  \end{bmatrix};
  \begin{bmatrix}
    B_{1} & B_{2} \\
   B_{3} & B_{4} \\
     \end{bmatrix};
     \begin{bmatrix}
    C_{1} & C_{2} \\
   C_{3} & C_{4} \\
  \end{bmatrix}\right)
\]

\hspace{5mm}

\noindent such that the two matrices $M$ and $N$ are invertible:

\[
M=\begin{bmatrix}
    A_2 & A_3 & A_4\\
   B_2& B_3 & B_4 \\
   C_2& C_3& C_4\\
  \end{bmatrix}
  \text{,\hspace{2mm} }
N=\begin{bmatrix}
    A_1 & A_2 & A_4\\
   B_1& B_2 & B_4 \\
   C_1& C_2& C_4\\
  \end{bmatrix}. 
\]

\hspace{5mm}

\noindent Then $T$ has a decomposition as the sum of the 3 simple tensors defined below.
\end{theorem}

\begin{proof}
Define $S_1,S_2,T_1$ and $T_2$ by

\[S_{1}=A_{2}B_{3}C_{4}-A_{2}B_{4}C_{3}-A_{3}B_{2}C_{4}+A_{3}B_{4}C_{2}+A_{4}B_{2}C_{3}-A_{4}B_{3}C_{2},\]
\[S_{2}=A_{1}B_{2}C_{4}-A_{1}B_{4}C_{2}-A_{2}B_{1}C_{4}+A_{2}B_{4}C_{1}+A_{4}B_{1}C_{2}-A_{4}B_{2}C_{1},\]
\[T_{1}=A_{1}B_{3}C_{4}-A_{1}B_{4}C_{3}-A_{3}B_{1}C_{4}+A_{3}B_{4}C_{1}+A_{4}B_{1}C_{3}-A_{4}B_{3}C_{1},\]
\[T_{2}=A_{1}B_{2}C_{3}-A_{1}B_{3}C_{2}-A_{2}B_{1}C_{3}+A_{2}B_{3}C_{1}+A_{3}B_{1}C_{2}-A_{3}B_{2}C_{1}.\]

\hspace{5mm}

\noindent We can check by inspection that 

\[A_{2}T_{1}-A_{1}S_{1}+A_{4}T_{2}=A_{3}S_{2},\]
\[B_{2}T_{1}-B_{1}S_{1}+B_{4}T_{2}=B_{3}S_{2},\]
\[C_{2}T_{1}-C_{1}S_{1}+C_{4}T_{2}=C_{3}S_{2}.\]

Since $S_{1}=\det(M)\neq 0$ and $S_{2}=\det(N)\neq 0$, we can write $T=T_{1}+T_{2}+T_{3}$ where:

\begin{align*}
& T_{1}=\left(
  \begin{bmatrix}
    A_{2}T_{1}S_{1}^{-1} & A_{2}  \\
   0 & 0  \\
  \end{bmatrix};
\begin{bmatrix}
    B_{2}T_{1}S_{1}^{-1} & B_{2}  \\
   0 & 0  \\
  \end{bmatrix};
  \begin{bmatrix}
    C_{2}T_{1}S_{1}^{-1} & C_{2}  \\
   0 & 0  \\
  \end{bmatrix}\right),\\
& T_{2}=\left(
  \begin{bmatrix}
     0 & 0  \\
    A_{4}T_{2}S_{2}^{-1} & A_{4}  \\
  \end{bmatrix};
\begin{bmatrix}
     0 & 0  \\
    B_{4}T_{2}S_{2}^{-1} & B_{4}  \\
  \end{bmatrix};
  \begin{bmatrix}
     0 & 0  \\
    C_{4}T_{2}S_{2}^{-1} & C_{4}  \\
  \end{bmatrix}\right),\\
& T_{3}=\left(
  \begin{bmatrix}
     (A_{1}S_{1}-A_{2}T_{1})S_{1}^{-1} & 0  \\
    -(A_{1}S_{1}-A_{2}T_{1})S_{2}^{-1} & 0 \\
  \end{bmatrix};
\begin{bmatrix}
     (B_{1}S_{1}-B_{2}T_{1})S_{1}^{-1} & 0  \\
    -(B_{1}S_{1}-B_{2}T_{1})S_{2}^{-1} & 0 \\
  \end{bmatrix};
  \begin{bmatrix}
     (C_{1}S_{1}-C_{2}T_{1})S_{1}^{-1} & 0  \\
    -(C_{1}S_{1}-C_{2}T_{1})S_{2}^{-1} & 0  \\
  \end{bmatrix}\right).\\
\end{align*}
\end{proof}

If we try and generalize Theorem \ref{222quaternion} or Theorem \ref{232complex} to write an explicit decomposition for the $2\times 2\times 3$ quaternion case, we immediately run into problems with the operations needed to write $T$ in an appropriate form. In Theorem \ref{222quaternion} for example, we only needed real row and column operations. As mentioned in Section \ref{definitions}, we cannot freely apply row and column operations using quaternions, and can only use row operations where $\mathbb{H}$ is acting on the left, and column operations where $\mathbb{H}$ is acting on the right (that is, horizontal and lateral slice operations). Nonetheless, we can provide a basic bound on the rank even if we cannot write down an explicit decomposition. Let us start with a motivational example.

\begin{example}\normalfont
Consider the quaternion $2\times 2\times 3$ tensor
 \[
   T=\left(
  \begin{bmatrix}
1 & i & 0  \\
0 & -j & 1+i\\
  \end{bmatrix};
  \begin{bmatrix}
  0 & 1+j & 0 \\
 0&  i+k & 1+j\\
     \end{bmatrix}\right).
\]

\noindent We can apply rank-preserving row and column operations to reduce $T$ to a form that is easier to decompose into simple tensors.

\begin{align*}
T&\longrightarrow \left(
  \begin{bmatrix}
1 & k & 0  \\
0 & 1 & 1+i\\
  \end{bmatrix};
  \begin{bmatrix}
  0 & -1+j & 0 \\
 0&  -i+k & 1+j\\
     \end{bmatrix}\right)\text{\hspace{0.5cm}  using $C_2\rightarrow C_2j$}\\
&\longrightarrow \left(
  \begin{bmatrix}
1 & k & 0  \\
0 & 1 & 0\\
  \end{bmatrix};
  \begin{bmatrix}
  0 & -1+j & 1+i-j+k \\
 0&  -i+k & i-k\\
     \end{bmatrix}\right) \text{\hspace{0.5cm}  using $C_3\rightarrow C_3-C_2(1+i)+C_1(j+k)$}\\
&\longrightarrow \left(
  \begin{bmatrix}
1 & 3k & 0  \\
0 & 3 & 0\\
  \end{bmatrix};
  \begin{bmatrix}
  0 & -1-j+2k & 1+i-j+k \\
 0&  2-i+k & i-k\\
     \end{bmatrix}\right)\text{\hspace{0.5cm} using $C_2\rightarrow 3C_2\newline -C_3(-2+i-k)$}\\
& =   \left(\begin{bmatrix}
1 & 0 & 0  \\
0 & 0 & 0\\
  \end{bmatrix};
  \begin{bmatrix}
  0 & 0 & 0 \\
 0& 0 & 0\\
     \end{bmatrix}\right)+
 \left(\begin{bmatrix}
0 & 0 & 0  \\
0 & 0 & 0\\
  \end{bmatrix};
  \begin{bmatrix}
  0 & 0 & (-i-j)(i-k) \\
 0& 0 & i-k\\
     \end{bmatrix}\right)\\
&\hspace{5mm}+
 \left(\begin{bmatrix}
0 & 3k & 0  \\
0 & 3 & 0\\
  \end{bmatrix};
  \begin{bmatrix}
 0 &  k(2-i+k) & 0 \\
 0& 2-i+k & 0\\
     \end{bmatrix}\right)\\
\end{align*}
Therefore, $\rank(T)\leq 3$. \hfill $\Box$
\end{example}

\begin{theorem}\label{223quaternion}
Let $T$ be a $2\times 2\times 3$ or a $3\times 2\times 2$ quaternion tensor. Then $\rank(T)\leq 3$. 
\end{theorem}

\begin{proof}
Let $ T=(
A;B)$ where
\[
   A=
  \begin{bmatrix}
    \vec{a} & \vec{b} & \vec{c}\\
  \end{bmatrix}, \hspace{2mm}
   B=
  \begin{bmatrix}
    \vec{d} & \vec{e} & \vec{f}\\
  \end{bmatrix},
\]
and $\vec{a}, \ldots, \vec{f}$ are 2-dimensional column vectors. If either $A$ or $B$ has rank no greater than 1,  we have
\begin{center}
    $\rank(T)\leq \rank((A;0))+\rank((0;B))\leq 1+2=3$.
\end{center}
Therefore, let us assume that $\rank(A)=\rank(B)=2$. Without loss of generality, we can assume that $\rank(\vec{a},\vec{b})=2$. Then, by a column operation, we have
\[
   T=\left(
  \begin{bmatrix}
    \vec{a} & \vec{b} & \vec{c}\\
  \end{bmatrix};
  \begin{bmatrix}
    \vec{d} & \vec{e} & \vec{f}\\
  \end{bmatrix}\right)\longrightarrow
    \left(\begin{bmatrix}
    \vec{a} & \vec{b} & 0\\
  \end{bmatrix};
  \begin{bmatrix}
    \vec{d} & \vec{e} & \vec{g}\\
  \end{bmatrix}\right).
\]
If $\vec{g}=0$, then $\rank(T)\leq 3$ by Theorem \ref{222quaternion}. Let us assume then that $\vec{g}\not=0$, and without loss of generality, that $\rank(\vec{e},\vec{g})=2$  (since $B$ has full rank). By further column operations, we have
\[
   \left(
  \begin{bmatrix}
    \vec{a} & \vec{b} & 0\\
  \end{bmatrix};
  \begin{bmatrix}
    \vec{d} & \vec{e} & \vec{g}\\
  \end{bmatrix}\right)\longrightarrow
   \left(\begin{bmatrix}
    \vec{h} & \vec{b} & 0\\
  \end{bmatrix};
  \begin{bmatrix}
    0 & \vec{e} & \vec{g}\\
  \end{bmatrix}\right).
\]
Since $\rank(\vec{h},\vec{b})=2$, we can write

\[ \vec{e}=\vec{h}c_{1}+\vec{b}c_{2} \text{ , \hspace{2mm}} \vec{g}=\vec{h}c_{3}+\vec{b}c_{4}.\]
If $c_{3}\not=0$, then
\begin{align*}
  \left(\begin{bmatrix}
    \vec{h} & \vec{b} & 0\\
  \end{bmatrix};
  \begin{bmatrix}
    0 & \vec{e} & \vec{g}\\
  \end{bmatrix}\right)&=
  \left(\begin{bmatrix}
    \vec{h} & \vec{b} & 0\\
  \end{bmatrix};
  \begin{bmatrix}
    0 & \vec{h}c_{1}+\vec{b}c_{2} & \vec{h}c_{3}+\vec{b}c_{4}\\
  \end{bmatrix}\right)\\
  &\longrightarrow
     \left(\begin{bmatrix}
    \vec{h} & \vec{b} & 0\\
  \end{bmatrix};
  \begin{bmatrix}
    0 & \vec{b}c_{5} & \vec{h}c_{3}+\vec{b}c_{4}\\
  \end{bmatrix}\right),
\end{align*}
which has rank at most 3 (we can write it as the sum of 3 simple tensors defined using the $2\times 2$ slices in the lateral direction). If $c_{3}=0$, then $c_{4}\not= 0$ and we have
\begin{align*}
  \left(\begin{bmatrix}
    \vec{h} & \vec{b} & 0\\
  \end{bmatrix};
  \begin{bmatrix}
    0 & \vec{e} & \vec{g}\\
  \end{bmatrix}\right)
  &=  \left(\begin{bmatrix}
    \vec{h} & \vec{b} & 0\\
  \end{bmatrix};
  \begin{bmatrix}
    0 & \vec{h}c_{1}+\vec{b}c_{2} & \vec{b}c_{4}\\
  \end{bmatrix}\right)\\
  &\longrightarrow
     \left(\begin{bmatrix}
    \vec{h} & \vec{h}c_{6}+\vec{b}c_{7} & 0\\
  \end{bmatrix};
  \begin{bmatrix}
    0 & (\vec{h}c_{6}+\vec{b}c_{7})c_{8} & \vec{b}c_{4}\\
  \end{bmatrix}\right),
\end{align*}
using the column operation $C_2\rightarrow C_{2}+C_{1}k_{1}+C_{3}k_{3}$ and choosing $k_1,k_3$ appropriately (depending on whether $c_1=0$ or $c_2=0$ for example). Then the rank is at most 3 by the same argument as above.

Finally, the $3\times 2\times 2$ case can be proven in a similar fashion by simply applying a rotation of the tensor and working with $A^T$ and $B^T$ instead. 
\end{proof}

\begin{theorem}\label{232quaternion}
Let $T$ be a $2\times 3\times 2$ quaternion tensor. Then $\rank(T) \leq 3$.
\end{theorem}

\begin{proof}
Let
 \[
   T= (A;B;C)=\left(
  \begin{bmatrix}
a_{11} & a_{12}  \\
a_{21} & a_{22}  \\
  \end{bmatrix};
  \begin{bmatrix}
b_{11} & b_{12}  \\
b_{21} & b_{22}  \\
     \end{bmatrix};
      \begin{bmatrix}
c_{11} & c_{12}  \\
c_{21} & c_{22}  \\
     \end{bmatrix}\right).
\]

\vspace{3mm}

\noindent If either $A, B$ or $C$ is singular (and if so, we may assume that it's $A$), then by rank-preserving row and column operations, the tensor can be reduced to 

 \[
   T=\left(
  \begin{bmatrix}
1 & 0  \\
0 & 0  \\
  \end{bmatrix};
  \begin{bmatrix}
b_{11}^{\prime} & b_{12}^{\prime}  \\
b_{21}^{\prime} & b_{22}^{\prime}  \\
     \end{bmatrix};
      \begin{bmatrix}
c_{11}^{\prime} & c_{12}^{\prime}  \\
c_{21}^{\prime} & c_{22}^{\prime}  \\
     \end{bmatrix}\right).
\]

\vspace{3mm}

\noindent If $b_{22}^{\prime}=c_{22}^{\prime}=0$, then $T$ is the sum of 3 simple tensors defined using the 3 nonzero vectors in the lateral direction ($\vec{v}=(1,b_{11}',c_{11}')$ is one such vector). Otherwise if either $b_{22}'\neq 0$ or $c_{22}'\neq 0$, then by adding a frontal slice to another we can assume that $b_{22}'c_{22}'\neq 0$, and we have the following decomposition into simple tensors $T=T_{1}+T_{2}+T_{3}$:

\begin{align*}
 &  T_{1}=\left(
  \begin{bmatrix}
0 & 0  \\
0 & 0  \\
  \end{bmatrix};
      \begin{bmatrix}
b_{12}^{\prime}(b_{22}^{\prime})^{-1}b_{21}^{\prime}  & b_{12}^{\prime}  \\
b_{22}^{\prime}(b_{22}^{\prime})^{-1}b_{21}^{\prime} & b_{22}^{\prime}  \\
     \end{bmatrix};
  \begin{bmatrix}
0 & 0  \\
0 &0  \\
     \end{bmatrix}\right),\\
&   T_{2}=\left(
  \begin{bmatrix}
0 & 0  \\
0 & 0  \\
  \end{bmatrix};
  \begin{bmatrix}
0 & 0  \\
0 &0  \\
     \end{bmatrix};
      \begin{bmatrix}
c_{12}^{\prime}(c_{22}^{\prime})^{-1}c_{21}^{\prime}  & c_{12}^{\prime}  \\
c_{22}^{\prime}(c_{22}^{\prime})^{-1}c_{21}^{\prime} & c_{22}^{\prime}  \\
     \end{bmatrix}\right),\\
&   T_{3}= \left(
  \begin{bmatrix}
1 & 0  \\
0 & 0  \\
  \end{bmatrix};
      \begin{bmatrix}
b_{11}^{\prime}-b_{12}^{\prime}(b_{22}^{\prime})^{-1}b_{21}^{\prime}  &0  \\
0 & 0  \\
     \end{bmatrix};
  \begin{bmatrix}
c_{11}^{\prime}-c_{12}^{\prime}(c_{22}^{\prime})^{-1}c_{21}^{\prime}  &0  \\
0 & 0  \\
     \end{bmatrix}\right).
\end{align*}

\vspace{3mm}

\noindent In both cases we see that $\rank(T) \leq 3$. 

Therefore, we may assume that $A,B$ and $C$ are nonsingular. By performing row and column operations, we can assume that    

     \[
   T=\left(
  \begin{bmatrix}
1 & 0  \\
0 & 1  \\
  \end{bmatrix};
      \begin{bmatrix}
b_{11} & b_{12}  \\
b_{21} & b_{22}  \\
     \end{bmatrix};
  \begin{bmatrix}
c_{11} & c_{12}  \\
c_{21} & c_{22}  \\
     \end{bmatrix}\right).
     \]
     
\vspace{3mm}

\noindent If either the second or the third frontal slice of $T$ is diagonalizable (without loss of generality assume the second one is diagonlizable), then $T$ can be further reduce to 

\begin{align}\label{diagcase of 232}
      S=\left(
  \begin{bmatrix}
1 & 0  \\
0 & 1  \\
  \end{bmatrix};
      \begin{bmatrix}
b_{11}^{\prime} & 0  \\
0 & b_{22}^{\prime}  \\
     \end{bmatrix};
  \begin{bmatrix}
c_{11}^{\prime} & c_{12}^{\prime}  \\
c_{21}^{\prime} & c_{22}^{\prime}  \\
     \end{bmatrix}\right),
\end{align}

\noindent which has the following decomposition $ S=S_{1}+S_{2}+S_{3}$ with $\rank(S_i)\leq 1$:

\begin{align*}
&   S_{1}=\left(
  \begin{bmatrix}
0 & 0  \\
0 & 0  \\
  \end{bmatrix};
      \begin{bmatrix}
0 & 0  \\
0 & 0  \\
     \end{bmatrix};
  \begin{bmatrix}
1 & c_{12}^{\prime}  \\
c_{21}^{\prime} & c_{21}^{\prime}c_{12}^{\prime}  \\
     \end{bmatrix}\right),\\
&   S_{2}=\left(
  \begin{bmatrix}
1 & 0  \\
0 & 0  \\
  \end{bmatrix};
      \begin{bmatrix}
b_{11}^{\prime} & 0  \\
0 & 0  \\
     \end{bmatrix};
  \begin{bmatrix}
c_{11}^{\prime}-1 & 0  \\
0 & 0  \\
     \end{bmatrix}\right),\\
&   S_{3}=\left(
  \begin{bmatrix}
0 & 0  \\
0 & 1  \\
  \end{bmatrix};
      \begin{bmatrix}
0 & 0  \\
0 & b_{22}^{\prime}  \\
     \end{bmatrix};
  \begin{bmatrix}
0 & 0   \\
0 & c_{22}^{\prime}- c_{21}^{\prime} c_{12}^{\prime} \\
     \end{bmatrix}\right).
\end{align*}
Otherwise, neither of them is diagonlizable. Then by \cite[Theorem 6.3]{Zhang} we can apply unitary triangularization for the second frontal slice to reduce $T$ to 

   \[
   S=\left(
  \begin{bmatrix}
1 & 0  \\
0 & 1  \\
  \end{bmatrix};
      \begin{bmatrix}
b_{11}^{\prime} & 0  \\
b_{21}^{\prime} & b_{22}^{\prime} \\
     \end{bmatrix};
  \begin{bmatrix}
c_{11}^{\prime} & c_{12}^{\prime}  \\
c_{21}^{\prime} & c_{22}^{\prime}  \\
     \end{bmatrix}\right).
     \]
Since the second frontal slice of $S$ is not diagonalizable, $b_{11}^{\prime}$ and $b_{22}^{\prime}$ must be equivalent. So we can assume by similarity that $b_{11}^{\prime}=b_{22}^{\prime}\in \mathbb{C}$. By adding a real multiple of the first frontal slice to the second, we can assume further that $b_{11}^{\prime}=b_{22}^{\prime}=ai$ for some $a\in \mathbb{R}$. Since the second slice is not singular, we have $a\not=0$. Then we can assume by rescaling the second slice that

   \[
   S=\left(
  \begin{bmatrix}
1 & 0  \\
0 & 1  \\
  \end{bmatrix};
      \begin{bmatrix}
i & 0  \\
b_{21}^{\prime} & i \\
     \end{bmatrix};
  \begin{bmatrix}
c_{11}^{\prime} & c_{12}^{\prime}  \\
c_{21}^{\prime} & c_{22}^{\prime}  \\
     \end{bmatrix}\right).
     \]
Denote $b_{21}^{\prime}=q_{0}+q_{1}i+q_{2}j+q_{3}k$ for some $q_{0}, q_{1}, q_{2},q_{3}\in\mathbb{R}$. We have

\[
  \begin{bmatrix}
1 & 0 \\
\frac{q_{3}j-q_{2}k}{2} & 1  \\
     \end{bmatrix}
    \begin{bmatrix}
i &  0 \\
b_{21}^{\prime} & i \\
     \end{bmatrix}
       \begin{bmatrix}
1 & 0  \\
-\frac{q_{3}j-q_{2}k}{2} & 1  \\
     \end{bmatrix}=
         \begin{bmatrix}
i & 0  \\
q_{0}+q_{1}i & i \\
     \end{bmatrix}.
\]
Since  the second slice is not diagonalizable by the assumption, ${q_{0}+q_{1}i}\not=0$. So we have
\[
  \begin{bmatrix}
1 & 0  \\
0 & \frac{1}{q_{0}+q_{1}i}  \\
     \end{bmatrix}
    \begin{bmatrix}
i & 0  \\
q_{0}+q_{1}i & i \\
     \end{bmatrix}
       \begin{bmatrix}
1 & 0  \\
0 & q_{0}+q_{1}i  \\
     \end{bmatrix}=
         \begin{bmatrix}
i & 0  \\
1 & i \\
     \end{bmatrix}.
\]
If we apply the same operations as above to the tensor $S$, we get the resulting tensor

   \[
   {S}^{(1)}=\left(
  \begin{bmatrix}
1 & 0  \\
0 & 1  \\
  \end{bmatrix};
      \begin{bmatrix}
i & 0  \\
1 & i \\
     \end{bmatrix};
  \begin{bmatrix}
c_{11}^{\prime\prime} & c_{12}^{\prime\prime}  \\
c_{21}^{\prime\prime} & c_{22}^{\prime\prime}  \\
     \end{bmatrix}\right).
     \]
Similarly, since the third frontal slice is not diagonalizable by the assumption, it can also be written as 

\begin{center}
      $\begin{bmatrix}
c_{11}^{\prime\prime} & c_{12}^{\prime\prime}  \\
c_{21}^{\prime\prime} & c_{22}^{\prime\prime}  \\
     \end{bmatrix}
     =P\begin{bmatrix}
m & 1  \\
0 & m \\
     \end{bmatrix}P^{-1}$
\end{center}
for some invertible matrix $P\in M_{n}(\mathbb{H})$ and some $m\in\mathbb{C}$. Therefore, by adding a multiple of the first frontal slice to the third as well as multiplying the third frontal by a real constant, $S^{(1)}$ can be further reduced to
   \[
   {S}^{(2)}=\left(
  \begin{bmatrix}
1 & 0  \\
0 & 1  \\
  \end{bmatrix};
      \begin{bmatrix}
i & 0  \\
1 & i \\
     \end{bmatrix};
P\begin{bmatrix}
i & 1  \\
0 & i \\
     \end{bmatrix}P^{-1}\right).
     \]
     Denote $P=
      \begin{bmatrix}
a & b  \\
c & d  \\
     \end{bmatrix}$ and decompose $P=P_{1}^{-1}P_{2}$ according to the following situations.
     
\begin{itemize}
    \item [(1)] If $a=0$, then define $ P_{1}=
           \begin{bmatrix}
0 & 1  \\
1 & 0  \\
     \end{bmatrix}$
      and 
      $P_{2}=\begin{bmatrix}
c & d  \\
0 & b  \\
     \end{bmatrix}
$.
\item [(2)]  If $a\not=0$, then define $  P_{1}=
           \begin{bmatrix}
1 & 0  \\
-ca^{-1} & 1  \\
     \end{bmatrix}$
      and 
      $P_{2}=\begin{bmatrix}
a & b  \\
0 & d-ca^{-1}b  \\
     \end{bmatrix}
$.
\end{itemize}
Then in either case, we have

   \[
   S^{(3)}=P_{1}{S}^{(2)}P_{1}^{-1}=\left(
  \begin{bmatrix}
1 & 0  \\
0 & 1  \\
  \end{bmatrix};
      P_{1}\begin{bmatrix}
i & 0  \\
1 & i \\
     \end{bmatrix}P_{1}^{-1};
P_{2}\begin{bmatrix}
i & 1  \\
0 & i \\
     \end{bmatrix}P_{2}^{-1}\right).
     \]

\noindent For (1), we can easily see that both the second and the third frontal slices of $S^{(3)}$ are  upper triangular. So, $S^{(3)}$ is the sum of 3 simple tensors defined using the 3 nonzero vectors in the lateral direction.\newline

\noindent For (2), we can compute

   \[
      P_{1}\begin{bmatrix}
i & 0  \\
1 & i \\
     \end{bmatrix}P_{1}^{-1}= \begin{bmatrix}
1 & 0  \\
-ca^{-1} & 1  \\
  \end{bmatrix}\begin{bmatrix}
i & 0  \\
1 & i \\
     \end{bmatrix}
     \begin{bmatrix}
1 & 0  \\
ca^{-1} & 1  \\
  \end{bmatrix}=\begin{bmatrix}
1 & 0  \\
1-ca^{-1}i+ica^{-1} & 1  \\
     \end{bmatrix}
          \]

and 
\begin{align*}
      P_{2}\begin{bmatrix}
i & 1  \\
0 & i \\
     \end{bmatrix}P_{2}^{-1}&= 
     \begin{bmatrix}
a & b  \\
0 & d-ca^{-1}b  \\
  \end{bmatrix}\begin{bmatrix}
i & 1  \\
0 & i \\
     \end{bmatrix}
     \begin{bmatrix}
a^{-1} & -a^{-1}b(d-ca^{-1}b)^{-1}  \\
0 & (d-ca^{-1}b)^{-1}  \\
  \end{bmatrix}\\
  &=\begin{bmatrix}
aia^{-1} & -aia^{-1}b(d-ca^{-1}b)^{-1}+(a+bi)(d-ca^{-1}b)^{-1}  \\
0 & (d-ca^{-1}b)i(d-ca^{-1}b)^{-1} \\
     \end{bmatrix}.
\end{align*}
Denote $c_{1}=1-ca^{-1}i+ica^{-1}$, $ b_{1}=-aia^{-1}b(d-ca^{-1}b)^{-1}+(a+bi)(d-ca^{-1}b)^{-1}$ and $d_{1}=(d-ca^{-1}b)$. We can write 
          
   \[
   S^{(3)}=\left(
  \begin{bmatrix}
1 & 0  \\
0 & 1  \\
  \end{bmatrix};
      \begin{bmatrix}
i & 0  \\
c_{1} & i \\
     \end{bmatrix};
\begin{bmatrix}
aia^{-1} & b_{1}  \\
0 & d_{1}i{d_{1}}^{-1} \\
     \end{bmatrix}\right).
     \]          
Note that $c_{1}\not=0$ by the assumption. Let $P_{3}=  \begin{bmatrix}
1 & 0  \\
0 & c_{1}  \\
  \end{bmatrix}$ and we have 
\begin{align*}
     S^{(4)}=P_{3}S^{(3)}P_{3}^{-1}&=\left(
  \begin{bmatrix}
1 & 0  \\
0 & 1  \\
  \end{bmatrix};
      \begin{bmatrix}
i & 0  \\
1 & c_{1}ic_{1}^{-1} \\
     \end{bmatrix};
\begin{bmatrix}
aia^{-1} & b_{1}c_{1}^{-1}  \\
0 & (c_{1}d_{1})i(c_{1}d_{1})^{-1} \\
     \end{bmatrix}\right)\\
    &=\left(
  \begin{bmatrix}
1 & 0  \\
0 & 1  \\
  \end{bmatrix};
      \begin{bmatrix}
i & 0  \\
1 & i_{1} \\
     \end{bmatrix};
\begin{bmatrix}
i_{2} & b_{1}c_{1}^{-1} \\
0 & i_{3} \\
     \end{bmatrix}\right),
\end{align*}
where $i_{1}=c_{1}ic_{1}^{-1}$, $i_{2}=aia^{-1}$ and $i_{3}=(c_{1}d_{1})i(c_{1}d_{1})^{-1}$ are all equivalent to $i$. Since the second frontal slice of $S^{(4)}$ is not diagonalizable, we have $i_{1}\not=-i$. So there exists $0\not=t\in \mathbb{R}$ such that $$\alpha:=-(i+ti_{2})\not= ( {i}_{1}+ti_{3}):=\beta.$$
Add the second frontal slice to $t$ times the third frontal slice, then $S^{(4)}$ is equivalent to

\begin{align*}
     S^{(5)}&=\left(
  \begin{bmatrix}
1 & 0  \\
0 & 1  \\
  \end{bmatrix};
      \begin{bmatrix}
i & 0  \\
1 & i_{1} \\
     \end{bmatrix};
\begin{bmatrix}
i+ti_{2} & tb_{1}c_{1}^{-1}  \\
1 & {i}_{1}+ti_{3} \\
     \end{bmatrix}\right)\\
     &=\left(
  \begin{bmatrix}
1 & 0  \\
0 & 1  \\
  \end{bmatrix};
      \begin{bmatrix}
i & 0  \\
1 & i_{1} \\
     \end{bmatrix};
\begin{bmatrix}
-\alpha & \gamma  \\
1 & \beta \\
     \end{bmatrix}\right),
\end{align*}
where $\gamma=tb_{1}c_{1}^{-1}$. We will show that the third frontal slice of $S^{(5)}$ is diagonalizable, so that $S^{(5)}$ reduces to the situation of Equality \ref{diagcase of 232}. Consider the following quaternion equation
\begin{equation*}\label{equation}
    x^{2}+\alpha x+x\beta-\gamma=0.
\end{equation*}
\noindent Since $\alpha\not=\beta$, according to \cite[Theorem 2.3.1]{huang2013quaternion} there exists a solution $x$ with $Re(x)\not=0$. Therefore, we have

\[
  \begin{bmatrix}
1 & -x  \\
0 & 1  \\
     \end{bmatrix}
    \begin{bmatrix}
-\alpha & \gamma  \\
1 & \beta \\
     \end{bmatrix}
       \begin{bmatrix}
1 & x  \\
0 & 1  \\
     \end{bmatrix}=
         \begin{bmatrix}
-\alpha-x & 0 \\
1 & \beta+x \\
     \end{bmatrix}.
\]
\noindent Note that both $\alpha$ and $\beta$ are pure imaginary numbers and $Re(x)\not=0$. We have
\begin{center}
    $Re(-\alpha-x)\not=Re(\beta+x)$,
\end{center}
which deduces that $-\alpha-x$ and $\beta+x$ must be non-equivalent. It follows that $         \begin{bmatrix}
-\alpha-x & 0 \\
1 & \beta+x \\
     \end{bmatrix}$ is diagonalizable, and thus so is $\begin{bmatrix}
-\alpha & \gamma  \\
1 & \beta \\
     \end{bmatrix}$ by similarity.

\end{proof}

Since we can view a $2\times 2\times 2$ quaternion tensor as a special case of the $2\times 2\times 3$ case, it is clear that a bound of 3 is the best possible bound since there already exist examples of tensors that have rank 3.

\section{The $2\times 3\times 3$ and $3\times 2\times 3$ cases}

We wish to bound the rank of a quaternion tensor with size $2\times 3\times 3$ (and subsequently the $3\times 3\times 2$ case) using the results of the previous sections. In order to do this, we will first need the following auxiliary lemma.

\begin{lemma}\label{sing}
Let $A,B\in M_{n}(\mathbb{H})$ be $n\times n$ matrices with entries in $\mathbb{H}$. If $A$ is invertible, then there exists an $x_{0}\in \mathbb{H}$ such that $x_{0}A+B$ is singular.
\end{lemma}

\begin{proof}
By \cite[Theorem 5.3]{Zhang}, every $n \times n$ quaternion matrix has at least one left eigenvalue in $\mathbb{H}$, which means that we can always choose $x_{0}\in \mathbb{H}$ such that $x_{0}I+BA^{-1}$ is singular. Therefore, $x_{0}A+B$ is also singular.
\end{proof}

\begin{theorem}\label{233quaternion}
Let $T$ be a $2\times 3\times 3$ or a $3\times 3\times 2$ quaternion tensor. Then $\rank(T)\leq 4$.
\end{theorem}
\begin{proof}
Let
 \[
   T=\left(
         \begin{bmatrix}
a_{11} & a_{12}  & a_{13} \\
b_{11} & b_{12}  & b_{13} \\
     \end{bmatrix};
  \begin{bmatrix}
a_{21} & a_{22}  & a_{23} \\
b_{21} & b_{22}  & b_{23} \\
     \end{bmatrix};
  \begin{bmatrix}
a_{31} & a_{32}  & a_{33} \\
b_{31} & b_{32}  & b_{33} \\
  \end{bmatrix}\right).
\]

\vspace{3mm}

\noindent We can write $T=T_{1}+T_{2}$, where 

\begin{align*}
&   T_{1}=\left(
      \begin{bmatrix}
a_{11} & a_{12}  & a_{13} \\
0 & 0  & 0 \\
     \end{bmatrix};
  \begin{bmatrix}
a_{21} & a_{22}  & a_{23} \\
0 & 0  & 0 \\
     \end{bmatrix};
 \begin{bmatrix}
a_{31} & a_{32}  & a_{33} \\
0 & 0  & 0 \\
  \end{bmatrix}\right),\\
&   T_{2}=\left(
      \begin{bmatrix}
0 & 0  & 0 \\
b_{11} & b_{12}  & b_{13} \\
     \end{bmatrix};
  \begin{bmatrix}
0 & 0  & 0 \\
b_{21} & b_{22}  & b_{23} \\
     \end{bmatrix};
      \begin{bmatrix}
0 & 0  & 0 \\
b_{31} & b_{32}  & b_{33} \\
  \end{bmatrix}\right).
\end{align*}
Consider the matrices
 \[
   A=
  \begin{bmatrix}
a_{11} & a_{12}  & a_{13} \\
a_{21} & a_{22}  & a_{23} \\
a_{31} & a_{32}  & a_{33} \\
  \end{bmatrix},\ \ \ 
  B=
   \begin{bmatrix}
b_{11} & b_{12}  & b_{13} \\
b_{21} & b_{22}  & b_{23} \\
b_{31} & b_{32}  & b_{33} \\
  \end{bmatrix}.
\]
Notice that $\rank(T_{1})=\rank(A)$ and $\rank(T_{2})=\rank(B)$. If both $A$ and $B$ have rank at most 2, then
\[\rank(T)\leq \rank(T_{1})+\rank(T_{2})\leq 2+2=4.\]

Therefore, without loss of generality, let us assume that $\rank(A)=3$. By Lemma \ref{sing}, there exists $x_{0}\in \mathbb{H}$ such that

 \[
   C=x_{0}A+B=
  \begin{bmatrix}
c_{11} & c_{12}  & c_{13} \\
c_{21} & c_{22}  & c_{23} \\
c_{31} & c_{32}  & c_{33} \\
  \end{bmatrix}
\]
is singular. Therefore $\rank(C)\leq 2$ and we can assume that
\[
   C=
  \begin{bmatrix}
c_{11} & c_{12}  & c_{13} \\
c_{21} & c_{22}  & c_{23} \\
c_{31} & c_{32}  & c_{33} \\
  \end{bmatrix}\longrightarrow
   \begin{bmatrix}
c_{11} & c_{12}  & 0 \\
c_{21} & c_{22}  & 0 \\
c_{31} & c_{32}  & 0 \\
  \end{bmatrix}.
\]

\vspace{3mm}

By rank-preserving row and column operations, we can write

\begin{align*}
   T&=\left(
      \begin{bmatrix}
a_{11} & a_{12}  & a_{13} \\
b_{11} & b_{12}  & b_{13} \\
     \end{bmatrix};
  \begin{bmatrix}
a_{21} & a_{22}  & a_{23} \\
b_{21} & b_{22}  & b_{23} \\
     \end{bmatrix};
  \begin{bmatrix}
a_{31} & a_{32}  & a_{33} \\
b_{31} & b_{32}  & b_{33} \\
  \end{bmatrix}\right),\\
&\longrightarrow
   \left(
       \begin{bmatrix}
a_{11} & a_{12}  & a_{13} \\
c_{11} & c_{12}  & c_{13} \\
     \end{bmatrix};
  \begin{bmatrix}
a_{21} & a_{22}  & a_{23} \\
c_{21} & c_{22}  & c_{23} \\
     \end{bmatrix};
      \begin{bmatrix}
a_{31} & a_{32}  & a_{33} \\
c_{31} & c_{32}  & c_{33} \\
  \end{bmatrix}\right),\\
&\longrightarrow
   \left(
        \begin{bmatrix}
a_{11} & a_{12}  & a_{13}^{\prime} \\
c_{11} & c_{12}  & 0 \\
     \end{bmatrix};
  \begin{bmatrix}
a_{21} & a_{22}  & a_{23}^{\prime} \\
c_{21} & c_{22}  & 0 \\
     \end{bmatrix};
     \begin{bmatrix}
a_{31} & a_{32}  & a_{33}^{\prime} \\
c_{31} & c_{32}  & 0 \\
  \end{bmatrix}\right).
\end{align*}

\vspace{3mm}

\noindent If $a_{13}^{\prime}=a_{23}^{\prime}=a_{33}^{\prime}=0$, then we immediately have $\rank(T)\leq 4$. Otherwise, we may that assume $a_{13}^{\prime}\not =0$. By column operations, the tensor can be reduced to 

 \[
  \left(
       \begin{bmatrix}
0 & 0  & a_{13}^{\prime} \\
c_{11} & c_{12}  & 0 \\
     \end{bmatrix};
  \begin{bmatrix}
a_{21}^{\prime} & a_{22}^{\prime}  & a_{23}^{\prime} \\
c_{21} & c_{22}  & 0 \\
     \end{bmatrix};
      \begin{bmatrix}
a_{31}^{\prime} & a_{32}^{\prime}  & a_{33}^{\prime} \\
c_{31} & c_{32}  & 0 \\
  \end{bmatrix}\right).
\]

\vspace{3mm}

\noindent If $c_{11}= c_{12}=0$, then $T$ is the sum of a $2\times 2\times 3$ tensor and a simple tensor, so we again have $\rank(T)\leq 4$ by Theorem \ref{223quaternion}. Therefore, by switching columns if necessary, we can assume that $c_{12}\not =0$. After performing further column operations we can reduce $T$ to 

\[
   \left(
   \begin{bmatrix}
0 & 0  & a_{13}^{\prime} \\
0 & c_{12}  & 0 \\
     \end{bmatrix};
  \begin{bmatrix}
a_{21}^{\prime\prime} & a_{22}^{\prime}  & a_{23}^{\prime} \\
c_{21}^{\prime} & c_{22}  & 0 \\
     \end{bmatrix};
       \begin{bmatrix}
a_{31}^{\prime\prime} & a_{32}^{\prime}  & a_{33}^{\prime} \\
c_{31}^{\prime} & c_{32}  & 0 \\
  \end{bmatrix}\right).
\]

\vspace{3mm}

\noindent If $a_{31}^{\prime\prime}=a_{21}^{\prime\prime}=0$, then $T$ is the sum of 4 simple tensors defined from the 4 nonzero $1\times 3$ vectors in the lateral direction, and we immediately have $\rank(T)\leq 4$. We can therefore assume that $a_{31}^{\prime\prime},a_{21}^{\prime\prime}\neq 0$  (by adding one matrix to the other if only one was nonzero). Now the remaining tensor can be decomposed as the sum of $T_{1}, T_{2}, T_{3}$ and  $T_{4}$ defined below, showing that $\rank(T)\leq 4$:

\begin{align*}
& T_{1}=\left(
      \begin{bmatrix}
0 & 0  & 0 \\
0 & c_{12}  & 0 \\
     \end{bmatrix};
  \begin{bmatrix}
0 & 0  & 0 \\
0 & c_{22}-c_{21}^{\prime}(a_{21}^{\prime\prime})^{-1}a_{22}^{\prime}   & 0 \\
     \end{bmatrix};
       \begin{bmatrix}
0 & 0  & 0 \\
0 & c_{32}-c_{31}^{\prime}(a_{31}^{\prime\prime})^{-1}a_{32}^{\prime}   & 0 \\
  \end{bmatrix}\right),\\
 & T_{2}= \left(
        \begin{bmatrix}
0 & 0  & 0\\
0 & 0 & 0 \\
     \end{bmatrix};
  \begin{bmatrix}
0 & 0  & 0 \\
0& 0  & 0 \\
     \end{bmatrix};
     \begin{bmatrix}
a_{31}^{\prime\prime} & a_{31}^{\prime\prime}(a_{31}^{\prime\prime})^{-1}a_{32}^{\prime}   & 0 \\
c_{31}^{\prime} & c_{31}^{\prime}(a_{31}^{\prime\prime})^{-1}a_{32}^{\prime}  & 0 \\
  \end{bmatrix}
\right),\\
& T_{3}=\left(
      \begin{bmatrix}
0 & 0  & 0\\
0 & 0 & 0 \\
     \end{bmatrix};
  \begin{bmatrix}
a_{21}^{\prime\prime} & a_{21}^{\prime\prime}(a_{21}^{\prime\prime})^{-1}a_{22}^{\prime}   & 0 \\
c_{21}^{\prime} & c_{21}^{\prime}(a_{21}^{\prime\prime})^{-1}a_{22}^{\prime}  & 0 \\
  \end{bmatrix};
      \begin{bmatrix}
0 & 0  & 0 \\
0& 0  & 0 \\
     \end{bmatrix}\right),\\
&  T_{4}= \left(
      \begin{bmatrix}
0 & 0  & a_{13}^{\prime} \\
0 & 0  & 0 \\
     \end{bmatrix};
  \begin{bmatrix}
0 & 0  & a_{23}^{\prime} \\
0 & 0  & 0 \\
     \end{bmatrix};
       \begin{bmatrix}
0 & 0  & a_{33}^{\prime} \\
0 & 0  & 0 \\
  \end{bmatrix}\right).
\end{align*}

\vspace{3mm}

Finally, we can repeat the above arguments for the $3\times 3\times 2$ case using the transpose of the matrices involved.
\end{proof}

It is not clear whether the bound of 4 is the best possible for the $2\times 3\times 3$ and $3\times 3\times 2$ cases. On the other hand, we can show that the tensor

\[  
T = \left(
      \begin{bmatrix}
1 & 0  & 0 \\
0 & 1  & 0 \\
0 & 0  & 1 \\
     \end{bmatrix};
       \begin{bmatrix}
0 & 0  & 1\\
0 & 1  & 0 \\
0& 0  & 0 \\
     \end{bmatrix}\right)
\]

\vspace{3mm}

\noindent has rank 4 by Lemma \ref{dia2} and Theorem \ref{323quaternion} below. The proof of the $3\times 2\times 3$ case however will require a more delicate argument than the one used in the $2\times 3\times 3$ case. 

\begin{theorem}\label{323quaternion}
Let $T$ be a $3\times 2\times 3$ quaternion tensor. Then $\rank(T)\leq 4$.
\end{theorem}
\begin{proof}
Let 
\[  
  T=(A;B)=\left(
      \begin{bmatrix}
a_{11} & a_{12}  & a_{13} \\
a_{21} & a_{22}  & a_{23} \\
a_{31} & a_{32}  & a_{33} \\
     \end{bmatrix};
       \begin{bmatrix}
b_{11} & b_{12}  & b_{13} \\
b_{21} & b_{22}  & b_{23} \\
b_{31} & b_{32}  & b_{33} \\
     \end{bmatrix}\right).
\]

\vspace{3mm}

\noindent If either $A$ or $B$ is singular (and without loss of generality assume that it's $A$), then we can use row operations to reduce the tensor to 

\[  
\left(
      \begin{bmatrix}
a_{11} & a_{12}  & a_{13} \\
a_{21} & a_{22}  & a_{23} \\
0 & 0  & 0 \\
     \end{bmatrix};
       \begin{bmatrix}
b_{11} & b_{12}  & b_{13} \\
b_{21} & b_{22}  & b_{23} \\
b_{31} & b_{32}  & b_{33} \\
     \end{bmatrix}\right),
\]

\vspace{3mm}

\noindent which is the sum of a $2\times 2\times 3$ tensor (using the first 2 horizontal slices) and a simple tensor defined by the row vector $(b_{31}, b_{32} ,b_{33})$. Then by Theorem \ref{223quaternion}, $\rank(T) \leq 3+1\leq 4$.
For the same reason, if any horizontal or lateral slice of the tensor does not have maximal rank, then $\rank(T)\leq 4$. For example, the first lateral slice of $T$ is the matrix

\[
\begin{bmatrix}
a_{11} & b_{11}  \\
a_{21} & b_{21} \\
a_{31} & b_{31}  \\
\end{bmatrix}.
\]

\vspace{3mm}

\noindent If it has rank 1, then $T$ can be written as the sum of a $3\times 2\times 2$ tensor and a simple tensor, proving that $\rank(T)\leq 4$. The maximal rank of any horizontal or lateral slice is 2.

Let us therefore assume that both $A$ and $B$ are nonsingular, and that any horizontal and lateral slice has rank 2. This means that there is at least one nonzero entry in every row and column of $A$ and $B$. It also means that we cannot have more than one row or column of zeros in any horizontal or lateral slice. 

Start by performing rank-preserving row operations to simplify the first column. Note that we cannot have equal first column vectors for $A$ and $B$ since this would contradict the rank assumption on the first lateral slice:

\[  
  T=\left(
      \begin{bmatrix}
a_{11} & a_{12}  & a_{13} \\
a_{21} & a_{22}  & a_{23} \\
a_{31} & a_{32}  & a_{33} \\
     \end{bmatrix};
       \begin{bmatrix}
b_{11} & b_{12}  & b_{13} \\
b_{21} & b_{22}  & b_{23} \\
b_{31} & b_{32}  & b_{33} \\
     \end{bmatrix}\right)\longrightarrow
       \left(\begin{bmatrix}
1 & a_{12}^{\prime}  & a_{13}^{\prime} \\
0 & a_{22}^{\prime}  & a_{23}^{\prime} \\
0 & a_{32}^{\prime}  & a_{33}^{\prime} \\
     \end{bmatrix};
       \begin{bmatrix}
0 & b_{12}^{\prime}  & b_{13}^{\prime} \\
1 & b_{22}^{\prime}  & b_{23}^{\prime} \\
0 & b_{32}^{\prime}  & b_{33}^{\prime} \\
     \end{bmatrix}\right).
\]

\vspace{3mm}

\noindent Similarly let us perform rank-preserving column operations to simplify the third row. We can also make $a_{12}'=0$ using the first column:
 
 \[  
  \longrightarrow
       \left(\begin{bmatrix}
1 & a_{12}^{\prime\prime}  & a_{13}^{\prime\prime} \\
0 & a_{22}^{\prime\prime}  & a_{23}^{\prime\prime} \\
0 & 0  & 1 \\
     \end{bmatrix};
       \begin{bmatrix}
0& b_{12}^{\prime\prime}  & b_{13}^{\prime\prime} \\
1 & b_{22}^{\prime\prime}  & b_{23}^{\prime\prime} \\
0 & 1  & 0 \\
     \end{bmatrix}\right)
     \longrightarrow
       \left(\begin{bmatrix}
1 & 0  & a_{13}^{\prime\prime} \\
0 & a_{22}^{\prime\prime}  & a_{23}^{\prime\prime} \\
0 & 0  & 1 \\
     \end{bmatrix};
       \begin{bmatrix}
0& b_{12}^{\prime\prime}  & b_{13}^{\prime\prime} \\
1 & b_{22}^{\prime\prime\prime}  & b_{23}^{\prime\prime} \\
0 & 1  & 0 \\
     \end{bmatrix}\right).
\]

\vspace{3mm}

\noindent Next we can use row operations to make $a_{23}'' = b_{12}''=0$:

\[
      \longrightarrow
       \left(\begin{bmatrix}
1 & 0  & a_{13}^{\prime\prime} \\
0 & a_{22}^{\prime\prime}  & 0 \\
0 & 0  & 1 \\
     \end{bmatrix};
       \begin{bmatrix}
0&  b_{12}^{\prime\prime} & b_{13}^{\prime\prime} \\
1 & b_{22}^{(4)}  & b_{23}^{\prime\prime} \\
0 & 1  & 0 \\
     \end{bmatrix}\right)
      \longrightarrow
       \left(\begin{bmatrix}
1 & 0  & a_{13}^{\prime\prime\prime} \\
0 & a_{22}^{\prime\prime}  & 0 \\
0 & 0  & 1 \\
     \end{bmatrix};
       \begin{bmatrix}
0&  0 & b_{13}^{\prime\prime} \\
1 & b_{22}^{(4)}  & b_{23}^{\prime\prime} \\
0 & 1  & 0 \\
     \end{bmatrix}\right).
\]

\vspace{3mm}

\noindent One final column operation can be used to make $a_{13}'''=0$. We can also scale the second row since $a_{22}''\neq 0$ (by the rank assumptions):

\[\longrightarrow
       \left(\begin{bmatrix}
1 & 0  & 0 \\
0 & a_{22}^{\prime\prime}  & 0 \\
0 & 0  & 1 \\
     \end{bmatrix};
       \begin{bmatrix}
0&  0 & b_{13}^{\prime\prime} \\
1 & b_{22}^{(4)}  & b_{23}^{\prime\prime\prime} \\
0 & 1  & 0 \\
     \end{bmatrix}\right)
        \longrightarrow
       \left(\begin{bmatrix}
1 & 0  & 0 \\
0 & 1  & 0 \\
0 & 0  & 1 \\
     \end{bmatrix};
       \begin{bmatrix}
0&  0 & b_{13}^{\prime\prime} \\
(a_{22}^{\prime\prime})^{-1} & (a_{22}^{\prime\prime})^{-1}b_{22}^{(4)}  & (a_{22}^{\prime\prime})^{-1}b_{23}^{\prime\prime\prime} \\
0 & 1  & 0 \\
     \end{bmatrix}\right).
     \]
 
 \vspace{3mm}
 
\noindent Let us relabel the entries by $w,x,y$ and $z$ to write the resulting tensor as   
     
\[
        \left(\begin{bmatrix}
1 & 0  & 0 \\
0 & 1  & 0 \\
0 & 0  & 1 \\
     \end{bmatrix};
       \begin{bmatrix}
0&  0 & w \\
x & y  & z \\
0 & 1  & 0 \\
     \end{bmatrix}\right).
\]

 \vspace{3mm}

Notice that if the second matrix (with entries $w,x,y$ and $z$) were diagonalizable, then the resulting tensor would have rank 3. This means that if we can find a vector $\vec{v}=(e,f,g)$ such that the matrix  

\[
       \begin{bmatrix}
0&  0 & w \\
x+e & y+f  & z+g \\
0 & 1  & 0 \\
     \end{bmatrix}
\]

\vspace{3mm}
 
\noindent is diagonalizable, then the tensor

\[
        \left(\begin{bmatrix}
1 & 0  & 0 \\
0 & 1  & 0 \\
0 & 0  & 1 \\
     \end{bmatrix};
       \begin{bmatrix}
0&  0 & w \\
x+e & y+f  & z+g \\
0 & 1  & 0 \\
     \end{bmatrix}\right)
     \]
     
\vspace{3mm}
      
\noindent would have rank 3, and $T$ could be written as the sum of a simple tensor (defined by $\vec{v}$) and a rank 3 tensor, showing again that $\rank(T)\leq 4$. Therefore, the problem reduces to choosing quaternions $x,y$ and $z$ such that the matrix

\[
       M=\begin{bmatrix}
0&  0 & w \\
x & y  & z \\
0 & 1  & 0 \\
     \end{bmatrix}
\]

 \vspace{3mm}
 
\noindent (where $w\not=0$ by the rank assumptions) is diagonalizable.

Set $w=a+bi+cj+dk$, and let us choose $x=u+vi$ to be a nonzero complex number and $y=z=0$. By Lemma \ref{adjoint}, the resulting matrix $M$ is diagonalizable if and only if its complex adjoint $\chi_M$ is. The ajoint of $M$ is the complex matrix

\[
       \chi_M=\begin{bmatrix}
0&  0 & a+bi &  0 & 0& c+di\\
u+vi&  0 & 0 &  0 & 0& 0\\
0&  1& 0 &  0 & 0 & 0\\
0&  0& -c+di &  0 & 0 & a-bi\\
0&  0 & 0 &  u-vi & 0 & 0\\
0&  0 & 0 &  0 & 1 & 0\\
     \end{bmatrix}.
\]

\vspace{3mm}
 
\noindent Its characteristic polynomial is 
\[ p_M(\lambda)={\lambda}^{6}-2(au-bv){\lambda}^{3}+(u^{2}+v^{2})(a^{2}+b^{2}+c^{2}+d^{2}).\]

\noindent Since $w\neq 0$, at least one of $a,b,c$ or $d$ is nonzero, and so $a^{2}+b^{2}+c^{2}+d^{2}\neq 0$. Choose $u,v\neq 0$ in $\mathbb{R}$ such that $au-bv=0$ (there are infinitely many choices here). Then $C=(u^{2}+v^{2})(a^{2}+b^{2}+c^{2}+d^{2})>0$ and 
\[p_M(\lambda)={\lambda}^{6}+ (u^{2}+v^{2})(a^{2}+b^{2}+c^{2}+d^{2})\]
has distinct roots given by $\sqrt[6]{C}\zeta_6^ki$ for $1\leq k\leq 6$, where $\zeta_6$ is a primitive sixth root of unity. Therefore $\chi_M$ is diagonalizable, completing the proof.
\end{proof}

\section{The $3\times 3\times 3$ case}

For real or complex $3\times 3\times 3$ tensors, it is known that the maximal rank is $5$  (see \cite[Section 3.4]{SMS2}). However, a general decomposition of such tensors into $5$ simple tensors is not provided. Although there are several subcases considered in \cite[Section 3.4]{SMS2}, we will only provide a decomposition for the main subcase (namely \cite[Equation 3.4.2]{SMS2}).

\begin{theorem}
Let $T$ be a complex $3\times 3\times 3$ tensor. Suppose that through a sequence of rank-preserving row and column operations, $T$ can be reduced to the form
\begin{align*}
T=\left(
  \begin{bmatrix}
    A_{11} & 0 & 1  \\
   0 & A_{22} & 0  \\
   1 &  0 & 0  \\
  \end{bmatrix};
\begin{bmatrix}
    B_{11} & 0 & 0  \\
   0 & B_{22} & 1  \\
   0 &  1 & 0  \\
  \end{bmatrix};
  \begin{bmatrix}
    C_{11} & C_{12} & 0  \\
    C_{21} & C_{22} & 0  \\
   0 & 0 & 0  \\
  \end{bmatrix}\right).
  \end{align*}

\hspace{5mm}

\noindent where $A_{11},B_{22},C_{22}\neq 0$ and $R =A_{11}C_{12}+B_{11}C_{22}\neq 0$. Then $T$ has a decomposition as the sum of the 5 simple tensors defined below.
\end{theorem}
\begin{proof}

Let $S=C_{11}C_{22}-C_{12}C_{21}$. Let us define $T_1,\ldots,T_5$ by

\begin{align*}
& T_{1}=\left(
  \begin{bmatrix}
    0 & 0 & 0  \\
   0 & 0 & 0  \\
   1 &  -B_{22}^{-1}A_{22} & 0  \\
  \end{bmatrix};
\begin{bmatrix}
    0 & 0 & 0  \\
   0 & 0 & 0  \\
   0 & 0 & 0  \\
  \end{bmatrix};
  \begin{bmatrix}
    0 & 0 & 0  \\
   0 & 0 & 0  \\
   0 & 0 & 0  \\
  \end{bmatrix}\right),\\
& T_{2}=\left(
  \begin{bmatrix}
    0 & 0 & 0  \\
   0 & 0 & 0  \\
   0 & 0 & 0  \\
  \end{bmatrix};
\begin{bmatrix}
    0 & 0 & 0  \\
   0 & 0 & 0  \\
   0 & 0 & 0  \\
  \end{bmatrix};
  \begin{bmatrix}
     (C_{12}C_{22}^{-1})C_{21} & (C_{12}C_{22}^{-1})C_{22} & (C_{12}C_{22}^{-1})(-SR^{-1})  \\
   C_{21} & C_{22} & -SR^{-1}  \\
   0 & 0 & 0  \\
  \end{bmatrix}\right),\\
& T_{3}=\left(
  \begin{bmatrix}
     0 & 0 & 0  \\
   0 & A_{22} & 0  \\
   0 & B_{22}^{-1}A_{22} & 0  \\
  \end{bmatrix};
\begin{bmatrix}
     0 & 0 & 0  \\
   0 & B_{22} & 0  \\
   0 & 1 & 0  \\
  \end{bmatrix};
  \begin{bmatrix}
     0 & 0 & 0  \\
   0 & 0 & 0  \\
   0 & 0 & 0  \\
  \end{bmatrix}\right).\\
& T_{4}=\left(
  \begin{bmatrix}
     A_{11} & 0 & A_{11}A_{11}^{-1}  \\
   0 & 0 & 0  \\
   0 & 0 & 0  \\
  \end{bmatrix};
\begin{bmatrix}
     B_{11} & 0 & B_{11}A_{11}^{-1}  \\
   0 & 0 & 0  \\
   0 & 0 & 0  \\
  \end{bmatrix};
  \begin{bmatrix}
     SC_{22}^{-1} & 0 & SC_{22}^{-1}A_{11}^{-1}  \\
   0 & 0 & 0  \\
   0 & 0 & 0  \\
  \end{bmatrix}\right).\\
  & T_{5}=\left(
  \begin{bmatrix}
     0 & 0 & 0  \\
   0 & 0 & 0  \\
   0 & 0 & 0  \\
  \end{bmatrix};
\begin{bmatrix}
     0 & 0 & -B_{11}A_{11}^{-1}  \\
   0 & 0 & 1  \\
   0 & 0 & 0  \\
  \end{bmatrix};
  \begin{bmatrix}
     0 & 0 & -B_{11}A_{11}^{-1}SR^{-1}  \\
   0 & 0 & SR^{-1}  \\
   0 & 0 & 0  \\
  \end{bmatrix}\right),\\
\end{align*}
 Note that some entries in $T_2$ and $T_4$ have not been simplified to be makes it clear that $\rank(T_i)\leq 1$. It is not difficult to verify that $T=T_{1}+T_{2}+T_{3}+T_{4}+T_{5}$, proving the result.
\end{proof}

Proceeding as before, we now use the tensor bounds from previous sections to provide a bound on the rank for the $3\times 3\times 3$ quaternion case. The technique from Theorem \ref{323quaternion} does not easily generalize here (the characteristic polynomial would be considerably more complicated for example), but we can still put a basic bound on the rank of a $3\times 3\times 3$ tensor using Lemma \ref{sing}. The authors do not know whether this bound could be improved on using a more refined method. 

\begin{theorem}
Let $T$ be a $3\times 3\times 3$ quaternion tensor. Then $\rank(T)\leq 6$.
\end{theorem}
\begin{proof}
Let

 \[
   T=\left(
      \begin{bmatrix}
a_{11} & a_{12}  & a_{13} \\
b_{11} & b_{12}  & b_{13} \\
c_{11} & c_{12}  & c_{13} \\
     \end{bmatrix};
  \begin{bmatrix}
a_{21} & a_{22}  & a_{23} \\
b_{21} & b_{22}  & b_{23} \\
c_{21} & c_{22}  & c_{23} \\
     \end{bmatrix};
       \begin{bmatrix}
a_{31} & a_{32}  & a_{33} \\
b_{31} & b_{32}  & b_{33} \\
c_{31} & c_{32}  & c_{33} \\
  \end{bmatrix}\right),
\]

\vspace{3mm}

\noindent and write $T$ as the sum $T_{1}+T_{2}+T_{3}$, where 

 \[
   T_{1}=\left(
       \begin{bmatrix}
a_{11} & a_{12}  & a_{13} \\
0 & 0  & 0 \\
0 & 0  & 0 \\
     \end{bmatrix};
  \begin{bmatrix}
a_{21} & a_{22}  & a_{23} \\
0 & 0  & 0 \\
0 & 0  & 0 \\
     \end{bmatrix};
       \begin{bmatrix}
a_{31} & a_{32}  & a_{33} \\
0 & 0  & 0 \\
0 & 0  & 0 \\
  \end{bmatrix}\right),
\]
 \[
   T_{2}=\left(
      \begin{bmatrix}
0 & 0  & 0 \\
b_{11} & b_{12}  & b_{13} \\
0 & 0  & 0 \\
     \end{bmatrix}
  \begin{bmatrix}
0 & 0  & 0 \\
b_{21} & b_{22}  & b_{23} \\
0 & 0  & 0 \\
     \end{bmatrix};
       \begin{bmatrix}
0 & 0  & 0 \\
b_{31} & b_{32}  & b_{33} \\
0 & 0  & 0 \\
  \end{bmatrix}\right),
\]
 \[
   T_{3}=\left(
       \begin{bmatrix}
0 & 0  & 0 \\
0 & 0  & 0 \\
c_{11} & c_{12}  & c_{13} \\
     \end{bmatrix}
  \begin{bmatrix}
0 & 0  & 0 \\
0 & 0  & 0 \\
c_{21} & c_{22}  & c_{23} \\
     \end{bmatrix};
      \begin{bmatrix}
0 & 0  & 0 \\
0 & 0  & 0 \\
c_{31} & c_{32}  & c_{33} \\
  \end{bmatrix}\right).
\]

\vspace{3mm}

\noindent Consider matrices

 \[
   A=
  \begin{bmatrix}
a_{11} & a_{12}  & a_{13} \\
a_{21} & a_{22}  & a_{23} \\
a_{31} & a_{32}  & a_{33} \\
  \end{bmatrix},\ \ \ 
  B=
   \begin{bmatrix}
b_{11} & b_{12}  & b_{13} \\
b_{21} & b_{22}  & b_{23} \\
b_{31} & b_{32}  & b_{33} \\
  \end{bmatrix}, \ \ \
    C=
   \begin{bmatrix}
c_{11} & c_{12}  & c_{13} \\
c_{21} & c_{22}  & c_{23} \\
c_{31} & c_{32}  & c_{33} \\
  \end{bmatrix}.
\]

\vspace{3mm}

\noindent Notice that $\rank(T_{1})=\rank(A)$,  $\rank(T_{2})=\rank(B)$, and $\rank(T_{3})=\rank(C)$, so if $A$, $B$ and $C$ all have rank at most 2, then
\[ \rank(T)\leq \rank(T_{1})+\rank(T_{2})+\rank(T_{3})\leq 2+2+2=6.\]
Therefore, without loss of generality, we can assume that $\rank(A)=3$. By Lemma \ref{sing}, there exists $x_{0}\in \mathbb{H}$ such that

 \[
   D=x_{0}A+C=
  \begin{bmatrix}
d_{11} & d_{12}  & d_{13} \\
d_{21} & d_{22}  & d_{23} \\
d_{31} & d_{32}  & d_{33} \\
  \end{bmatrix}
\]

\vspace{3mm}

\noindent is singular. Using rank-preserving row operations, we can reduce $T$ to 

 \[
   T
\longrightarrow
   \left(
    \begin{bmatrix}
a_{11} & a_{12}  & a_{13} \\
b_{11} & b_{12}  & b_{13} \\
d_{11} & d_{12}  & d_{13} \\
     \end{bmatrix}
  \begin{bmatrix}
a_{21} & a_{22}  & a_{23} \\
b_{21} & b_{22}  & b_{23} \\
d_{21} & d_{22}  & d_{23} \\
     \end{bmatrix};
      \begin{bmatrix}
a_{31} & a_{32}  & a_{33} \\
b_{31} & b_{32}  & b_{33} \\
d_{31} & d_{32}  & d_{33} \\
  \end{bmatrix} \right).
\]

\vspace{3mm}

\noindent Notice that the first 2 horizontal slices $T_{1}+T_{2}$ can viewed as a $2\times 3\times  3$ tensor, which has rank at most 4 by Theorem \ref{233quaternion}, and the last slice 

 \[
   \left(
 \begin{bmatrix}
0 & 0  & 0 \\
0 & 0  & 0 \\
d_{11} & d_{12}  & d_{13} \\
     \end{bmatrix};
  \begin{bmatrix}
0 & 0  & 0 \\
0 & 0  & 0 \\
d_{21} & d_{22}  & d_{23} \\
     \end{bmatrix};
       \begin{bmatrix}
0 & 0  & 0 \\
0 & 0  & 0 \\
d_{31} & d_{32}  & d_{33} \\
  \end{bmatrix} \right)
\]

\vspace{3mm}

\noindent has the same rank as the matrix $D$. Therefore, we have 

\[\rank(T)\leq \rank(T_{1}+T_{2})+\rank(D)\leq 4+2=6.\]
\end{proof}

\section*{Acknowledgements}

We thank the anonymous referee for their valuable comments.

\bibliographystyle{plain}
\bibliography{Library}

\end{document}